\documentclass[english]{article}
\usepackage{amsmath}
\usepackage{amssymb}
\usepackage{esint}
\usepackage{amsthm}\usepackage{amsfonts}
\usepackage{geometry}
\newtheorem{theorem}{Theorem}[section]
\newtheorem{lemma}{Lemma}[section]

\theoremstyle{remark}
\newtheorem{remark}{Remark}[section]
\theoremstyle{definition}

\begin{document}

\title{Positivity of transition probabilities of infinite-dimensional diffusion processes on ellipsoids}

\author{Oxana  Manita}


\maketitle

{\sc Abstract:} We consider diffusion processes in Hilbert spaces with 
constant non-degenerate diffusion operators and show
that, under broad assumptions
on the drift,  the transition  probabilities of the process are positive on ellipsoids
 associated with the diffusion operator. 
This is an infinite-dimensional
analogue of positivity of densities of transition probabilities.
Our results apply to diffusions corresponding to
stochastic partial differential equations. 

$\quad$
~

{\sc Keywords:} diffusion process in Hilbert space; SPDE; support
of distribution; positive density; mild solution; variational solution;
Kolmogorov equation.

$\quad$
~

{\sc Author's address:}  
Faculty of Mechanics and Mathematics, 
Moscow State University, Moscow 119991, 
Russia;  o.manita@lambda.msu.ru.

\section{Introduction }

Let us consider the  stochastic differential equation (SDE)
\begin{equation}
dX_{t}=dW_{t}+(AX_{t}+F(X_{t})dt),\quad X_{0}=\eta \label{eq:sde}
\end{equation}
in a Hilbert space $H$, where $W_t$ is an $H$-valued Wiener 
process with  covariance operator $Q$, having  eigenvectors $\{e_i\}_{i\in \mathbb{N}}$ 
and eigenvalues 
$\{q_i\}_{i\in \mathbb{N}}$,  
 and the corresponding Kolmogorov equation
\begin{equation}
\partial_{t}\mu_{t}=\frac{1}{2}q_{i}\partial_{e_{i}e_{i}}^{2}\mu_{t}-
\partial_{e_{i}}(b^{i}(x)\mu_{t}),\quad\mu_{0}=
\mbox{Law}(\eta)
\label{eq:CP-1}
\end{equation}
for the distributions $\mu_{t}$ of the diffusion process $X_{t}$. 
Here  $b^{i}=\langle A+F,e_{i}\rangle $,   $A$ is
a linear (possibly unbounded) operator and $F$ is some function
on~$H$. Equations of such a form correspond to stochastic
partial differential equations (SPDEs). In typical cases $A$
is an elliptic differential operator.

It is well-known that if the coefficients of the equation are regular
enough and the diffusion matrix is non-degenerate, then the transition
probabilities of the finite-dimensional diffusion process have strictly
positive densities with respect to  Lebesgue measure (see \cite{Regularity}).
  In the non-degenerate case this property is usually derived from the Harnack inequality or from the Girsanov
theorem. Another powerful approach is provided by the seminal result of Strook and Varadhan 
 \cite{StrSup}. In the finite-dimensional case they give a full description of the 
 support of the distribution of the diffusion process
$$
dX_t = \sigma (X_t) \circ dW_t + F(X_t)dt, \quad X_0 = x,
$$
where the SDE is written in the form of Stratonovich. Namely, they showed that 
the support of the distribution of $X_t$ coincides with the closure in the space of continuous functions 
of the set of solutions to the appropriate control problem: the Wiener process is replaced by a smooth path --
 control -- and the SDE turns into an ODE in the Hilbert space). More precisely, they showed that 
 $\mbox{supp }\mbox{Law}(X_t)=\mathcal{S}_t$, where 
\begin{equation}
\mathcal{S}_{t}=\overline{\bigl\{ y_{t}:\, u \mbox{\,\,is piecewise constant}\,\,\mbox{and }
\,\dot{y}=\sigma(y_{s})u+F(y_{s}),\quad y_{0}=x\bigr\} }.
\label{eq:control_pr}\end{equation}
We emphasize that this results doesn't require non-degeneracy of the diffusion (and is interesting mostly 
in the degenerate case).

However, in the infinite-dimensional case the situation is different.
First of all, in the infinite-dimensional case there is no Lebesgue measure.
Therefore, we consider the following property: the measure of
every open set is strictly positive. In the finite-dimensional case
this  holds in case of existence of a strictly positive
density with respect to Lebesgue measure. Even for the best studied
class of measures in the infinite-dimensional spaces -- Gaussian measures
-- this property is not quite trivial (see \cite[Theorem 3.5.1]{gaussian_measures}).
Positivity on open sets  is sometimes called irreducibility of the 
semigroup corresponding to the 
diffusion process (irreducibility of the generator of the process).  
Next, there is no exact analogue of Harnack's inequality in Hilbert spaces 
(for upper bounds see \cite{DaPratoHarnack});
the Girsanov theorem is applicable only in very special cases where
drifts take values in the Cameron--Martin spaces of the corresponding Wiener
processes. Moreover, there are no full analogues of the result of Strook and Varadhan. 
Hence the following question
arises: is the distribution of a non-degenerate diffusion process
in a Hilbert space at time $t$ positive on all open sets (at least
for processes with bounded drifts)? The answer
is positive for linear SDEs of the form
$$
dX_{t}=dW_{t}+AX_{t}dt,\quad X_{0}=x.
$$
This equation admits an explicit solution that is a Gaussian process.
However, in the general case  the solution to \eqref{eq:sde} is not
a Gaussian process. It needs not be even
absolutely continuous with respect to a Gaussian process.

Despite the fact that this question is of considerable interest for SPDEs,
only a few results in this direction  are known.
For some special equations (such as the stochastic
Navier--Stokes equation) this question was studied by diverse methods
(see \cite{AgrKSS,KuksinArmen} and \cite{PardouxMatt}).
We also mention the paper \cite{barbu},
where strict positivity in the above sense was established for the
invariant measure of the stochastic porous medium equation.
 
The problem in the general setting was considered in the book \cite{ErgZabczyk}
 for  Lipschitz continuous  perturbations $F$. 
  The positive result for non-degenerate constant diffusion operators is 
obtained in \cite[Theorem 7.4.2]{ErgZabczyk} by methods of the control theory, inspired 
by the ideas of  Strook and Varadhan \cite{StrSup}.
 However, in this approach it is impossible to drop the 
assumption of the Lipschitz continuity of $F$.

In this paper we study  the question of positivity of the
distribution of non-degenerate diffusion processes on open sets with purely probabilistic 
methods. We consider constant non-degenerate
diffusion operators and drifts that are bounded 
perturbations of linear operators and prove that at every positive time 
the distribution of such a process is
positive on every ellipsoid whose axis are given by the eigenvectors of the
diffusion operator. This means that the distribution  has full topological support 
in the weaker topology in which these ellipsoids are balls.  The main difference
of this result from the above mentioned result in \cite{ErgZabczyk} is that we don't 
assume that the nonlinear term $F$ is Lipschitz continuous.  
   Instead of this, we assume that 
the SDE and the corresponding Kolmogorov equation have unique solutions. 
This is a much milder assumption 
since typically SDEs with non-degenerate diffusions are more regular than ODEs.
Moreover,  due to  the fast development of the field and   new results on well-posedness, this assumption 
 is less and less restrictive.   
 The second difference consists in using purely probabilistic methods without references to the control theory. 

Let us proceed to exact statements.

Let $H$ be a separable real Hilbert space with inner product $\langle \cdot,\cdot\rangle $
and norm $\Vert{\cdot}\Vert$. Fix a positive self-adjoint operator $Q\colon\,H\rightarrow H$
with finite trace and eigenvalues $\{ q_{j}\} _{j\in\mathbb{N}}$.
Set
$$
\mbox{tr\,} Q:=\sum_{j=0}^{\infty}q_{j}<\infty.
$$
We assume that
\begin{equation}
1=q_{1}\geq q_{2}\geq\dots>0.\label{eq:q1}
\end{equation}
 We  assume that we are given an $H$-valued Wiener process $\bigl(W_{t},\, t\in\mathbb{R}_{+}\bigr)$
 on some probability space $\bigl(\Omega,\mathcal{F},\mathbb{P}\bigr)$  with covariance operator $Q$,
i.e. $$
\mathbb{E}\langle W_{t},u\rangle \langle W_{s},v\rangle =
\min\{ t,s\} \cdot\langle Qu,v\rangle .
$$
Let $\bigl(\mathcal{F}_{t},\, t\geq0\bigr)$ be the filtration
generated by this Wiener process. There exist an orthonormal system
$\{ e_{j}\} _{j\in\mathbb{N}}$ in $H$ (see \cite[Proposition 4.3]{Zabczyk})
and a countable set of independent one-dimensional standard Wiener
processes ${\displaystyle (\beta_{t}^{j},\, t\in\mathbb{R}_{+}),\,\, j\in\mathbb{N}}$ on
$\bigl(\Omega,\mathcal{F},\mathbb{P}\bigr)$ that are
$\bigl(\mathcal{F}_{t},\, t\geq0\bigr)$-adapted such that
\begin{equation}
W_{t}=\sum_{j=1}^{\infty}\sqrt{q_{j}}\beta_{t}^{j}e_{j},
\label{eq:wp}
\end{equation}
where the series converges in $L^2$. Define a weighted norm
on $H$ by $$
\Vert x\Vert _{Q}:=\langle Qx,x\rangle ^{1/2}=
\Bigl(\sum_{j=1}^{\infty}q_{j}x_{j}^{2}\Bigr)^{1/2},
\quad x_{j}:=\langle x,e_{j}\rangle
$$
and observe that $\Vert x\Vert_{Q}\leq  \Vert x\Vert $ for each $x\in H$
due to \eqref{eq:q1}. Given $a\in H$ and $R\in\mathbb{R}_{+}$,
set
\begin{align*}
K_{R}(a): & =\{ x\in H:\,\Vert x-a\Vert_{Q}\leq R\} ,
\quad U_{R}(a):=\{ x\in  H:\,\Vert x-a\Vert \leq R\} ;\\
 & \quad K_{R}(0)=:K_{R},\qquad U_{R}(0)=:U_{R}.
 \end{align*}
 The sets $K_{R}(a)$ will be called ellipsoids and the
sets $U_{R}(a)$ will be called balls. The ellipsoid $K_{R}(a)$
contains $U_{R}(a)$, but is not contained in any ball $U_{R^{'}}(a^{'})$
(contrary to the finite-dimensional case).

Let $\mathcal{B}(H)$ denote the $\sigma$-field
of all Borel sets in $H$. Let $\mathcal{P}_{\infty}(H)$ denote the set of all probability
measures on $(H,\mathcal{B}(H))$ with
finite moments of all orders. Let  $\mathcal{V}_{\infty}(H)$ denote  the set of all
$H$-valued random variables with finite moments of all orders.
Finally, let $\mathcal{FC}_{0}^{\infty}(H)$  denote the class of all functions of the form
$\phi(x)=\phi_{0}(x_{1},\ldots,x_{m})$
with some $m\in\mathbb{N}$, where $\phi_{0}$ is an  infinitely smooth function
with  compact support in $\mathbb{R}^{m}$.

\section{SDE with a  bounded  drift }

First we consider the case of a  bounded  drift. This case
is not only interesting in itself, but is also a basis for further
consideration.

Suppose that an $H$-valued random variable $\eta$ and a function $F:\,H\rightarrow H$
are given. 

 On the probability space $\bigl(\Omega,\mathcal{F},\mathbb{P}\bigr)$ consider the following SDE:
\begin{equation}
dX_{t}=dW_{t}+F(X_{t})dt,\qquad X_{0}=\eta.
\label{eq:sde-1}
\end{equation}
 An $\mathcal{F}_{t}$-adapted $ H$-valued process $\bigl(X_{t},\, t\in\mathbb{R}_{+}\bigr)$ is said
to be a strong solution to \eqref{eq:sde-1} if $\mathbb{P}$-a.s. for all
$t\geq0$
\begin{equation}
X_{t}=\eta+W_{t}+\int_{0}^{t}F(X_{s})ds,
\label{eq:diff-pr-1}
\end{equation}
where the last integral is a Bochner integral. In the sequel we shall
consider the distributions $(\mu_{t})_{t\geq0}$ of the
process $\bigl(X_{t},\, t\in\mathbb{R}_{+}\bigr)$, defined by
$$
\mu_{t}(C)=\mathbb{P}(X_{t}\in C),\quad C\in\mathcal{B}(H).
$$
 To the diffusion process \eqref{eq:sde-1} we associate the Cauchy
problem for its distributions
\begin{equation}
\partial_{t}\mu_{t}=\frac{1}{2}q_{i}\partial_{e_{i}e_{i}}^{2}\mu_{t}-
\partial_{e_{i}}(b^{i}(x)\mu_{t}),
\quad\mu_{0}=\nu=\mbox{Law}(\eta),
\label{eq:CP}
\end{equation}
where $b^{i}=\langle F,e_{i}\rangle $. Throughout the
paper we assume that summation over all repeated indices is taken.
A family of probability measures $(\mu_{t})_{t\geq0}$
is said to be a solution to \eqref{eq:CP} if the identity
$$
\int\phi(x)d\mu_{t}-\int\phi(x)d\nu=
\int_{0}^{t}\int\mathcal{L}\phi(x)d\mu_{s}ds,$$
where
$$
	\mathcal{L}\phi=\sum_{i=0}^{\infty}\frac{1}{2}q_i \partial^{2}_{e_i e_i}\phi+\sum_{i=0}^{\infty}b^i \partial_{e_i}\phi,
$$
holds for all $t\geq0$ and all test functions $\phi\in\mathcal{FC}_{0}^{\infty}(H)$.

Further we assume that

(i) $\eta$ is independent of $\bigl(W_{t},\, t\in\mathbb{R}_{+}\bigr)$ and
 $\eta\in\mathcal{V}_{\infty}\bigl(H\bigr)$;

(ii) the function $F$ is bounded, i.e.
$$
\sup_{x\in  H}\Vert F(x)\Vert=F_{*}<+\infty.
$$

(iii) The  equation \eqref{eq:sde-1} has a
 strong solution $X_{t}, t\geq0$ and  $X_{t}\in\mathcal{V}_{\infty}(H)$ for each $t\geq0$.
The problem \eqref{eq:CP} has a unique probability
solution.

Under assumption (iii) the distributions of
the process $X_t$ solve the Cauchy problem \eqref{eq:CP} (see \cite[Section 14.2.2]{Zabczyk}).
This one-to-one correspondence between equations enables us to switch
between probability representations and measures whenever it is convenient.

\begin{theorem} \label{pos_ell} 
Assume ${\rm (i)}$, ${\rm (ii)}$ and ${\rm (iii)}$
hold. Then, for any initial condition $\eta\in\mathcal{V}_{\infty}(H)$ and for  every $T>0$,
 the solution to (\ref{eq:CP}) is strictly
positive on every ellipsoid $K_{R}(a)$:
$$
\mu_{T}(K_{R}(a))>0,\,\,\mbox{or, equivalently, }\mathbb{P}(X_{T}\in K_{R}(a))>0.
$$
\end{theorem}

\begin{remark}\label{never} Equation  \eqref{eq:CP}
is meaningful for any nonnegative finite Borel initial measure $\nu$,
and then the solution is a finite nonnegative Borel measure and preserves
the total mass  $\nu(H)$ of the space. Hence the result
of Theorem \ref{pos_ell} is valid for the Cauchy problem \eqref{eq:CP}
with any finite nonnegative Borel initial measure $\nu$.
\end{remark}

\begin{remark}\label{weak} As it can be seen from the proof, in  ${\rm (iii)}$
instead of existence of a strong solution 
 it sufficies to 
assume only existence of a weak solution which possesses the Markov property. 
 In regular finite-dimensional cases existence of a weak solution, together with uniqueness 
 of distribution, ensures \cite{Krylov_Markov_sel} that it is a Markov process on its probability basis.  
Morever, existence of weak solution is closely related to the solvability of the corresponding 
martingale problem, which, in it's turn, is connected to the well-posedness of the Kolmogorov equation. 
However, the author doesn't know any precise analogues of these results in the 
infinite-dimensional setting. To the author's knowledge, similar results are proved under additional 
assumptions like $m$-dissipativity of the drift or for equations with initial data from a particular 
class (for example, see \cite{Beznea_Boboc_Roeckner}).

\end{remark}

\begin{proof} We split the proof into
several steps.

1. We prove that for each ellipsoid $K_{R}(a)$, each initial
distribution $\nu\in\mathcal{P}_{\infty}(H)$ and each $T>0$, there exists a time $t_{0}\in(0,T]$
such that at $t_{0}$ the solution to the Cauchy problem (\ref{eq:CP})
is strictly positive on $K_{R}(a)$:
\begin{equation}
\mu_{t_{0}}(K_{R}(a))=\mathbb{P}(X_{t_{0}}\in K_{R}(a))>0.
\label{eq:dense_t-1}
\end{equation}

2. We prove that, for each ellipsoid $K_{R}(a)$, there
exists $\tau=\tau(R)>0$ such that for any initial distribution
$\nu\in\mathcal{P}_{\infty}\bigl(H\bigr)$ one has 
$$
\mu_{t}(K_{R}(a))=\mathbb{P}(X_{t}\in K_{R}(a))>0\quad\forall t\in(0,\tau].
$$

3. We prove the assertion of the theorem, i.e., that 
$$
\mu_{t}(K_{R}(a))=\mathbb{P}(X_{t}\in K_{R}(a))>0\quad\forall t>0.
$$

\textbf{Step 1.} First, let us show that for each initial
 measure $\nu\in\mathcal{P}_{\infty}\bigl(H\bigr)$ that
is not Dirac's measure at zero and for each $T>0$, there exists $t_{0}\in(0,T${]}
such that $\mu_{t_{0}}(K_{R})>0$. It suffices to prove
this assertion for initial measures with $\mbox{supp}\,\nu\subset U_{N}\backslash K_{\delta}$
for some $N>\delta>0$. Indeed, assume that \eqref{eq:dense_t-1} holds
for every initial measure supported in $U_{N}\backslash K_{\delta}$.
The continuity of $\nu$ at zero yields that there is $\delta>0$
such that $\nu(H\backslash K_{\delta})>0$. Since
$$
U_{N}\subset U_{N+1}\quad\mbox{and }\quad\bigcup_{N=1}^{\infty}U_{N}=H,
$$
there is an index $N_{0}$ such that $\nu(U_{N_{0}}\backslash K_{\delta})>0$.
Define measures $\nu_{0}$ and $\nu^{\bot}$ by
$$
\nu_{0}(E)=\nu(E\cap(U_{N_{0}}\backslash K_{\delta})),
\quad\nu^{\bot}(E)=\nu(E\backslash(U_{N_{0}}\backslash K_{\delta})).
$$
Then $\nu=\nu_{0}+\nu^{\bot}$. Observe that equation \eqref{eq:CP}
is linear in measure, hence $\mu_{t}=\mu_{t}^{0}+\mu_{t}^{\bot}$,
where $\mu_{t}^{0}$, $\mu_{t}^{\bot}$ are solutions to \eqref{eq:CP}
with initial measures $\nu_{0}$ and $\nu^{\bot}$, respectively. By
Remark \ref{never}, \eqref{eq:dense_t-1} holds for
the family $(\mu_{t}^{0})_{t\geq0}$ with some $t_{0}\in(0,T]$,
thus
$$
\mu_{t_{0}}(K_{R})=
\mu_{t_{0}}^{0}(K_{R})+\mu_{t_{0}}^{\bot}(K_{R})\geq
\mu_{t_{0}}^{0}(K_{R})>0.
$$
Hence we can assume from the very beginning that the initial measure $\nu$ satisfies the condition 
$$
\mbox{supp}\,\nu\subset U_{N}\backslash K_{\delta}\quad\mbox{for some }N>\delta>0.
$$
In particular, $\nu$ can be an atomic measure outside zero. Fix $K_{R}=K_{R}(0)$ and $T>0$.
Let $\eta$ be an $H$-valued
random variable independent of $\bigl(W_{t},\, t\in\mathbb{R}_{+}\bigr)$
 such that $\mbox{Law}(\eta)=\nu$.

Let us show that there exists $t_{0}\in(0,T]$ such that $\mu_{t_{0}}(K_{R})>0$.
We argue by contradiction. Suppose that this is false and $\mu_{t}(K_{R})=0$
for all $t\in(0,T]$. Without loss of generality we can assume that $R<\delta$
and $\mu_{t}(K_{R})=0$ for all $t\in[0,T]$.
In particular, this means that $\mathbb{P}$-a.s. 
$\Vert X_{t} \Vert\geq\Vert X_{t}\Vert_{Q}\geq R$
for all $t\in[0,T]$.

Consider the one-dimensional stochastic process $\zeta_{t}=\Vert X_{t}\Vert ^{2}$.
It is a smooth function of the diffusion process \eqref{eq:diff-pr-1}
and its It{\^o}'s differential can be computed by using It{\^o}'s formula for
$H$-valued processes (see \cite[Theorem 4.32]{Zabczyk}):
$$
d\zeta_{t}=2\langle X_{t},dW_{t}\rangle +
(2\langle X_{t},F(X_{t})\rangle +\mbox{tr\,} Q)dt,
\quad\zeta_{0}=\Vert\eta \Vert^{2}.
$$
In order to simplify the first term in the differential, we observe
that the one-dimensional stochastic process $w=(w_{t},\, t\geq0)$ given by
$$
w_{t}=\int_{0}^{t}\frac{\langle X_{s},dW_{s}\rangle }{\Vert X_{s}\Vert_{Q}}
$$
is a continuous square-integrable $\mathcal{F}_{t}$-martingale
and (see \cite[Theorem 4.27]{Zabczyk}) its quadratic variation equals
$$
\ll w_{t}\gg=\int_{0}^{t}\Phi_{s}\, ds,
$$
where
$$
\Phi_{s}=\Bigl(\frac{X_{s}}{\Vert X_{s}\Vert_{Q}}Q^{1/2}\Bigr)\Bigl(\frac{X_{s}}{\Vert X_{s}\Vert_{Q}}Q^{1/2}\Bigr)^{*}=
\frac{1}{\Vert X_{s}\Vert_{Q}^{2}}\cdot\bigl(X_{s}\, Q^{1/2}\bigr)\bigl(X_{s}\, Q^{1/2}\bigr)^{*}=
\frac{\Vert X_{s}\Vert_{Q}^{2}}{\Vert X_{s}\Vert_{Q}^{2}}=1.
$$
Hence $\ll w_{t}\gg=t$. L{\'e}vy's characterization of the Brownian
motion (see \cite[Chapter 3, Theorem 3.16]{Karatzas}) yields that $w$
is an $\mathcal{F}_{t}$-adapted Wiener process. Thus,
\begin{align}
\zeta_{t} & =\zeta_{0}+\int_{0}^{t}v(\omega,s)dw_{s}+\int_{0}^{t}c(\omega,s)ds,\nonumber \\
 & v(\omega,t):=2\Vert X_{t}\Vert_{Q},\,\,\,
 c(\omega,t):=2\langle X_{t},F(X_{t})\rangle +\mbox{tr\,} Q.
 \label{eq:zeta}\end{align}
 Observe that $v(\omega,t)$ is also a progressively measurable
($\mathcal{F}_{t}$-adapted) process. Since
$$
c(\omega,t)\leq\mbox{tr}Q+\zeta_{t}+\Vert F\Vert^{2}_{\infty}=:\lambda+\zeta_{t},
$$
by the assumption $\zeta_{0}\leq N^{2}$ we have
\begin{equation}
\zeta_{t}\leq\zeta_{0}+\int_{0}^{t}v(\omega,s)dw_{s}+
\int_{0}^{t}(\lambda+\zeta_{s})ds\leq(N^{2}+T\lambda)+
\int_{0}^{t}v(\omega,s)dw_{s}+\int_{0}^{t}\zeta_{s}ds.
\label{eq:2}
\end{equation}
 Letting 
 $$\Psi_{t}:=(N^{2}+T\lambda)+\int_{0}^{t}v(\omega,s)dw_{s},
 $$
we obtain 
\begin{equation}
\zeta_{t}\leq\Psi_{t}+\int_{0}^{t}\zeta_{s}ds.
\label{eq:1-1}
\end{equation}
 Multiplying by $e^{-t}$, we obtain 
 $$
\frac{d}{dt}\biggl(e^{-t}\cdot\int_{0}^{t}\zeta_{s}ds\biggr)\leq e^{-t}\Psi_{t},
\quad\mbox{hence }\quad\int_{0}^{t}\zeta_{s}ds\leq\int_{0}^{t}e^{t-s}\Psi_{s}ds.
$$
 Plugging this estimate into \eqref{eq:2}, we arrive at 
\begin{equation}
0\leq\zeta_{t}\leq\Psi_{t}+\int_{0}^{t}e^{t-s}\Psi_{s}ds
\leq C(N,T)+\int_{0}^{t}v(\omega,s)dw_{s}+
\int_{0}^{t}e^{t-s}\int_{0}^{s}v(\omega,r)dw_{r}ds,
\label{eq:3}\end{equation}
where $C(N,T):=(N^{2}+T\lambda)(1+Te^{T})>0$.
Next, by the integration by parts formula (see \cite[Ex. 4.3]{OksSPE}), we have 
\begin{multline*}
\int_{0}^{t}e^{t-s}\int_{0}^{s}v(\omega,r)dw_{r}ds=
e^{t}\Bigl(-e^{-t}\int_{0}^{t}v(\omega,r)dw_{r}\Bigr)+
e^{t}\int_{0}^{t}e^{-s}d\int_{0}^{s}v(\omega,r)dw_{r}\\=
-\int_{0}^{t}v(\omega,r)dw_{r}+e^{t}\int_{0}^{t}e^{-s}v(\omega,s)dw_{s},
\end{multline*}
hence \eqref{eq:3} implies that for $t\in[0,T]$
\begin{equation}
\int_{0}^{t}e^{-s}v(\omega,s)dw_{s}\geq-C(N,T)\cdot e^{-t}\geq-C(N,T).
\label{eq:lower_bound}
\end{equation}
 By our assumption $v\geq2R$. Fix an arbitrary $t^{*}\in(0,t)$.
Define a random change of time
\begin{equation}
z_{t}:=\int_{0}^{t}e^{-2s}v^{2}(\omega,s)ds\geq t\cdot(2R)^{2}e^{-2T}.
\label{eq:time-change}
\end{equation}
 For each $\gamma\geq0$, set $\tau_{\gamma}:=\inf\{ s\geq0:\, z_{s}=\gamma\}$.
The paths of the process $z_{t}$ are continuous and
the process is bounded from below according to \eqref{eq:time-change}, hence $\tau_{\gamma}$
is a stopping time with respect to the filtration $(\mathcal{F}_{t},\, t\geq0)$. 
Moreover,  $\mathbb{P}(\tau_{\gamma}<+\infty)=1$ and $\tau_{\gamma}<t^{*}$
for each $\gamma<t^{*}\cdot(2R)^{2}e^{-2T}$ with $\mathbb{P}$-probability~$1$. 
The change of time theorem (\cite[Chapter 1, Par. 4, Theorem 3]{Skorohod})
implies that the stochastic process $y=(y_{\gamma},\,\gamma\geq0)$ given by
$$
y_{\gamma}:=\int_{0}^{\tau_{\gamma}}e^{-s}v(\omega,s)dw_{s}
$$
is a Wiener process with respect to the filtration $(\mathcal{F}_{\tau_{\gamma}},\,\gamma\geq0)$.
In particular, the random variable $y_{\gamma}$ has a strictly positive distribution 
density on the real line. On the other hand, $\int_{0}^{t}e^{-s}v(\omega,s)dw_{s}$
is an $\mathcal{F}_{t}$-martingale. It is well-known (see, for example,
\cite[Paragraph 7.2, Microtheorem 3]{Wenz}) that the martingale property
 holds not only for deterministic times, but also
for bounded stopping times: $\mathbb{P}$-a.s. one has 
$$
y_{\gamma}=\int_{0}^{\tau_{\gamma}}e^{-s}v_{s}dw_{s}=
\int_{0}^{\tau_{\gamma}\wedge t}e^{-s}v_{s}dw_{s}=
\mathbb{E}\Bigl(\int_{0}^{t}e^{-s}v_{s}dw_{s}|\,\mathcal{F}_{\tau_{\gamma}}\Bigr)\geq
-C(N,T),
$$
since $\tau_{\gamma}<t^*<t$. This contradiction means that there exists
$t_{0}\in(0,T]$ such that $\mu_{t_{0}}(K_{R})>0$.

Let us now proceed to non-centered ellipsoids. Fix $K_{R}(a)$
with a center $a\in  H$. Let us show that there is $t_{0}\in(0,T]$
such that the solution to (\ref{eq:CP}) is positive on $K_{R}(a)$
for very initial measure $\nu\neq\delta_{a}$. Fix $\nu\neq\delta_{a}$.

Consider the shift $L^{a}\colon\, H\to  H$ defined by
$$
L^{a}x=x+a.
$$
We recall that  the image of a measure
$\rho$ under the mapping $L^{a}$ 
is the measure $L_{*}^{a}\rho$ defined by $L_{*}^{a}\rho(E)=\rho(L^{a}(E)))$
for each measurable set $E\subset H$. 
Then it follows from the
definition that $L^{a}(K_{R})=K_{R}(a)$ and the measures
$\sigma_{t}=L_{*}^{a}\mu_{t}$ satisfy the equation 
$$
\partial_{t}\sigma_{t}=\frac{1}{2}q_{i}\partial_{e_{i}e_{i}}^{2}\sigma_{t}-
\partial_{e_{i}}(b^{i}(x-a)\sigma_{t}),
\quad\sigma_{0}=L_{*}^{a}\nu\neq\delta_{0},
$$
where $b^{i}(\cdot-a)=\langle F(\cdot-a),e_{i}\rangle$.
The drift term $F(\cdot-a)$. Therefore, by the assertion for centered
balls proved above in the case $\sigma_{0}\neq\delta_{0}$, there exists
$t_{0}\in(0,T]$ such that
$$
\mu_{t_{0}}(K_{R}(a))=
\mu_{t_{0}}(L^{a}(K_{R}))\overset{def}{=}
L_{*}^{a}\mu_{t_{0}}(K_{R})>0.
$$
 To complete the proof of this step, we consider $K_{R}(a)$
and $\nu=\delta_{a}$. Note that for $\varepsilon>0$ small enough
$$
K_{R/2}(a+\bar{\varepsilon})\subset K_{R}(a),\quad\bar{\varepsilon}=\varepsilon\cdot e_{1}\in  H.
$$
 Indeed, if $(x_{1}-a_{1}-\varepsilon)^{2}+\sum_{j=2}^{\infty}q_{j}(x_{j}-a_{j})^{2}\leq R^{2}/4$,
then
$$
\sum_{j=1}^{\infty}q_{j}(x_{j}-a_{j})^{2}\leq
2(x_{1}-a_{1}-\varepsilon)^{2}+2\varepsilon^{2}+\sum_{j=2}^{\infty}q_{j}(x_{j}-a_{j})^{2}\leq
R^{2}/2+2\varepsilon^{2}\leq R^{2}
$$
for $\varepsilon^{2}\leq R^{2}/4$. But \eqref{eq:dense_t-1} has already
been proved for $K_{R/2}(a+\bar{\varepsilon})$ and $\nu=\delta_{a}$,
i.e. $\mu_{t_{0}}(K_{R/2}(a+\varepsilon))>0$ for some
$t_{0}\in(0,T]$. By additivity $\mu_{t_{0}}(K_{R}(a))\geq\mu_{t_{0}}(K_{R/2}(a+\bar{\varepsilon)})>0$.

\textbf{Step 2.} Let us prove that for every ellipsoid $K_{R}(a)$,
there exists $\tau=\tau(R)>0$, depending only on $R$
and $\sup$-norm of $F$, such that \emph{for any} initial distribution
$\nu\in\mathcal{P}_{\infty}(H)$ one has 
$$
\mathbb{P}(X_{t}\in K_{R})>0\quad\mbox{for all }t\in(0,\tau(R)],
$$
where $X_{t}$ solves \eqref{eq:sde-1}.

The idea of the proof is quite simple: if the process with 
any  initial distribution at some time $t_{0}$ hits a small ellipsoid
with positive probability, then with positive
probability it stays in a larger ellipsoid during some time, and
this time is determined by the parameters of the ellipsoid. But it
has already been proven that during every small interval of time the
process $X_{t}$ hits every fixed ellipsoid (with positive probability)
at least once. The combination of these facts yields the assertion
of Step 2. Let us proceed to rigorous proofs.

Fix $X_{0}\in\mathcal{V}_{\infty}\bigl(H\bigr)$ and $K_{R}(a)$. Set
$$
\tau(R):=R\cdot\Bigl(6\cdot(1+\sup_{x\in  H}\Vert F(x)\Vert_{Q})\Bigr)^{-1}.
$$

\begin{lemma}\label{nu->dalee} Assume that $\mbox{supp}\,\nu\subset K_{R/2}(a)$.
Then
\begin{equation}
\mathbb{P}(X_{t}\in K_{R}(a))>0\quad\mbox{for all }t\in(0,\tau(R)].
\label{eq:taur}
\end{equation}
\end{lemma}
{\bf Proof of Lemma \ref{nu->dalee}.}  
Recall that 
$$X_{t}=X_{0}+W_{t}+\int_{0}^{t}F(X_{s})ds.
$$
Obviously, it suffices to show that for all $t\in(0,\tau(R)]$
$$
\mathbb{P}\Bigl(\Vert X_{t}-X_{0}\Vert_{Q}>R/2\Bigr)<1.
$$
 This follows from the properties of $H$-valued Wiener processes
and the definition of $\tau(R)$. Indeed,
\begin{multline}
\mathbb{P}\Bigl(\Bigl\Vert X_{t}-X_{0}\Bigr\Vert_{Q}>\frac{R}{2}\Bigr)=\mathbb{P}\Bigl(\Bigl\Vert W_{t}+
\int_{0}^{t}F(X_{s})ds\Bigr\Vert_{Q}>\frac{R}{2}\Bigr)\\
\leq\mathbb{P}\Bigl(\Vert W_{t}\Vert_{Q}>\frac{R}{4}\Bigr)+
\mathbb{P}\Bigl(\Bigl\Vert \int_{0}^{t}F(X_{s})ds\Bigr\Vert_{Q}>\frac{R}{4}\Bigr).
\label{eq:2ver}\end{multline}
By the properties of the Bochner integral and the definition of $\tau(R)$ we have
$$
\Bigl\Vert\int_{0}^{t}F(X_{s})ds\Bigr\Vert_{Q}\leq\int_{0}^{t}\Vert F(X_{s})\Vert_{Q}ds\leq
\tau(R)\cdot\sup_{x\in  H}\Vert F(x)\Vert_{Q}\leq\frac{R}{6}<\frac{R}{4},
$$
i.e. the second probability on the right-hand side of \eqref{eq:2ver}
equals zero. Hence
$$
\mathbb{P}(\Vert X_{t}-X_{0}\Vert_{Q}>\frac{R}{2})\leq\mathbb{P}(\Vert W_{t}\Vert_{Q}>\frac{R}{4})\leq
\mathbb{P}(\Vert W_{t} \Vert>\frac{R}{4}).
$$
 The distribution of $W_{t}$ at time $t$ is a centered $H$-valued
Gaussian random variable with variance $t\cdot Q$. By \cite[Theorem 3.5.1]{gaussian_measures}
the probability on the right-hand side of the last inequality is strictly
less than $1$. This completes the proof of Lemma \ref{nu->dalee}.
\hfill{}$\square$

Let us return to the proof of Step 2. Fix $\delta\in(0,\tau(R))$.
According to Step 1 there exists a time $t_{0}\in(0,\delta)$
such that ${\displaystyle \mu_{t_{0}}(K_{R/2}(a))>0}$.
By the Markov property
$$
X_{t}=X_{t_{0}}+W_{t}^{1}+\int_{t_{0}}^{t}F(X_{s})ds,\quad t\geq t_{0}
$$
where $W^{1}=(W_{t}-W_{t_{0}},\, t\geq t_{0})$ is also
a $Q$-Wiener process. By our choice of $t_{0}$ we have 
$$
\mathbb{P}(X_{t_{0}}\in K_{R/2}(a))=\mu_{t_{0}}(K_{R/2}(a))>0.
$$
 Arguing similarly to Step 1 and applying Lemma \ref{nu->dalee},
we obtain 
$$
\mathbb{P}(X_{t}\in K_{R}(a))=
\mu_{t}(K_{R}(a))>0
\quad\forall t\in[t_{0},t_{0}+\tau(R)].
$$
 In particular, this holds for all $t\in[\delta,\tau(R)]$,
but $\delta$ is an arbitrary number in $(0,\tau(R)]$, hence
\begin{equation}
\mathbb{P}(X_{t}\in K_{R}(a))=
\mu_{t}(K_{R}(a))>0
\quad\mbox{for all }t\in(0,\tau(R)].
\label{eq:mal_otr}
\end{equation}

\textbf{ Step 3.} Fix an arbitrary time $M$. Split the interval
$[0,M]$ into $n:=[M/\tau(R)]$ parts,
where $\tau(R)$ is defined by \eqref{eq:taur}:
$$
[0,M]=\bigcup_{i=0}^{n-1}[s_{i},s_{i+1}],\quad
 s_{j}=j\cdot\tau(R),\,\, j=0,\dots,n-1,\quad s_{n}=M.
$$
 By the previous step, for any initial data $\eta\in\mathcal{V}_{\infty}\bigl(H\bigr)$,
the assertion of Theorem \ref{pos_ell} holds on $[0,\tau(R)]\equiv[s_{0},s_{1}]$,
i.e. \eqref{eq:mal_otr}. Similarly to the Step 2, we have
$$
X_{t}=X_{s_{1}}+W_{t}^{2}+\int_{s_{1}}^{t}F(X_{s})ds,\quad t\geq s_{1}.
$$
 Application of the result of Step 2 gives that 
 $\mathbb{P}(X_{t}\in K_{R}(a))=\mu_{t}(K_{R}(a))>0$
for $t\in(s_{1},s_{2}]$. By induction we get
$$
\mu_{t}(K_{R}(a))>0\quad\mbox{for all }t\in(0,M].
$$
 This completes the proof of Theorem \ref{pos_ell}.
 \end{proof}
 
 \begin{remark}\label{Novikov} If $u(\omega,t)=e^{-t}v(\omega,t)$
is not separated from zero, then, generally speaking, \eqref{eq:lower_bound}
does not yield a contradiction. This can be shown by a simple
example (suggested by  A.A. Novikov). Consider $u(\omega,t)=\exp\{ w_{t}-t/2\} >0$
$\mathbb{P}$-a.s., where $w_{t}$ is a standard Wiener process on the real
line. Obviously, there is no positive $R$ such that $\mathbb{P}$-a.s. $u(\omega,s)\geq R$.
It{\^o}'s formula implies
$$
u(\omega,t)=1+\int_{0}^{t}u(\omega,s)dw_{s}>0,\quad
\mbox{hence\,\,}\int_{0}^{t}u(\omega,s)dw_{s}>-1\quad\mathbb{P}-\mbox{a.s.}
$$
\end{remark}

\begin{remark}\label{lip} The assumption (iii) is fulfilled, for example, 
if $F$ is Lipschitz continuous. Equation \eqref{eq:sde-1} has a
unique strong solution $X_{t}, t\geq0$ due to  \cite[Theorem 7.2]{Zabczyk} and 
 $X_{t}\in\mathcal{V}_{\infty}(H)$ for each $t\geq0$. The problem \eqref{eq:CP} has a 
 unique probability
solution by virtue of \cite[Theorem 1]{parab_inf} and \cite[Theorem 2.1]{AnApproach}.
However, Theorem  \ref{pos_ell} is in a sense stronger than \cite[Theorem 7.4.2]{ErgZabczyk}, mentioned in the 
Introduction, where irreducibility of the corresponding semigroup is demonstrated, because 
it does not require any continuity of the nonlinear perturbation. 
\end{remark}

\section{SDE with unbounded drift }

We now proceed to the general case -- SDE \eqref{eq:sde} with an
unbounded self-adjoint negative linear operator $A$:
\begin{equation}
dX_{t}=dW_{t}+(AX_{t}+F(X_{t}))dt,\quad X_{0}=\eta.
\label{eq:gen_eq}
\end{equation}
 Here, as above, $\bigl(W_{t},\, t\in\mathbb{R}_{+}\bigr)$ is a $Q$-Wiener process 
 on $\bigl(\Omega,\mathcal{F},\mathbb{P}\bigr)$ with the natural
filtration $(\mathcal{F}_{t},\, t\geq0)$.  Set $B(x)=Ax+F(x)$.

Let us now recall the concept of variational solution (see \cite{ConciseCourse}).

Consider the Banach space $V:=D((-A)^{1/2})$ equipped
with the graph norm of $(-A)^{1/2}$ and its dual space
$V^{*}$. Then  $(V,H,V^{*})$ is  a Gelfand
triple, i.e.  $V\subset  H\subset V^{*}$ and the embeddings are
continuous and dense. Let us consider the Friedrichs extension $A_1$
of $A$. Then $A_1\colon\, V\to V^{*}$ and $A_1$ is
also a densely defined negative self-adjoint operator (see, for example,
\cite[Theorem 2.23]{Kato}). Set
$B_1(\cdot):=A_1+F(\cdot)\colon\, V\to V^{*}$.
For notational simplicity, further we omit indices,
and $A$ will denote not only the operator, but also its Friedrichs extension,
and also $B(\cdot)=A+F(\cdot)$.

A continuous $H$-valued $\mathcal{F}_{t}$-adapted process $X=\bigl(X_{t},\, t\in[0,T]\bigr)$
is called a variational solution to \eqref{eq:gen_eq} if for its
$dt\times\mathbb{P}$-equivalence class $\hat{X}$ with some $\alpha\geq1$
we have
$\hat{X}\in
L^{\alpha}([0,T]\times\Omega,dt\times\mathbb{P};V)
\cap L^{2}([0,T]\times\Omega,dt\times\mathbb{P};H)$
and $\mathbb{P}$-a.s.
\begin{equation}
X_{t}=\eta+W_{t}+\int_{0}^{t}B(\bar{X}_{s})ds,\quad t\in[0,T],
\label{eq:var-sol}
\end{equation}
where $\bar{X}$ is any $\mathcal{F}_{t}$-adapted $V$-valued $dt\times\mathbb{P}$-version
of $\hat{X}$. Moreover, the integrand in \eqref{eq:var-sol} is automatically 
$H$-valued (see, for example, \cite[Remark 4.2.2]{ConciseCourse}).
Below we set $\alpha=2$.

Along with assumptions (i) and (ii) from the previous section, 
we shall need the following assumptions:

(iii') The
problem \eqref{eq:CP} has
a unique probability solution. The equation \eqref{eq:gen_eq} has a 
variational solution (see \cite{KRoz}) and
\begin{equation}
\mathbb{E}\sup_{t\in[0,T]}\Vert X_{t}\Vert^{2}<+\infty.
\label{eq:mom2}
\end{equation} 

(iv) The domain $D(A)\subset  H$ of the linear operator
$A$ is dense in $H$ and $A$ is self-adjoint and negative (i.e.
$\langle Ax,x\rangle \leq-\varepsilon\Vert x^{2} \Vert$ for some $\varepsilon>0$
and all $x\in H$).

The Hille--Yosida theorem (see, for example, \cite[Theorem 2.6]{KurtzEthier})
states that any linear operator $A$ with properties (iv) generates
a contracting strongly continuous semigroup $S_{t},\, t\in\mathbb{R}_{+}$
of linear transformations of $H$.

A continuous $\mathcal{F}_{t}$-adapted $H$-valued process $X=\bigl(X_{t},\, t\in[0,T]\bigr)$
is said to be a mild solution to \eqref{eq:gen_eq} (see, for example,
\cite{Zabczyk,ConciseCourse}) if $\mathbb{P}$-a.s. for all $t\in[0,T]$ one has 
\begin{equation}
X_{t}=S_{t}\eta+\int_{0}^{t}S_{t-s}I\, dW_{s}+\int_{0}^{t}S_{t-s}F(X_{s})ds.
\label{eq:diff-pr}
\end{equation}
 Here  $I$ is the
identity operator on $H$; the last integration is in Bochner's sense.

 The distributions
of the process $X_t$ solve \eqref{eq:CP} with $b^{i}=\langle B,e_{i}\rangle $
(see \cite[Section 14.2.2]{Zabczyk}).   As above, this one-to-one correspondence
enables us  to consider measures in placed of processes and vice versa, whenever this is convenient.

The  main result of this section is the following theorem. 

\begin{theorem} \label{lin_pos} Assume that ${\rm (i)}$, ${\rm (ii)}$, ${\rm (iii')}$
and ${\rm (iv)}$ hold. Then, for any initial condition $\eta\in\mathcal{V}_{\infty}(H)$ and 
for every  $t\in(0,T]$, the solution to
(\ref{eq:CP}) is strictly positive on each ellipsoid $K_{R}(a)$:
$$
\mu_{t}(K_{R}(a))>0,\,\,\mbox{or, equivalently, }\mathbb{P}(X_{t}\in K_{R}(a))>0.
$$
\end{theorem}

\begin{proof} The proof mainly repeats the proof of Theorem \ref{pos_ell}.
We consider  only the steps affected by the addition of the linear term.

 Arguing similarly to Step 1 of the proof of Theorem \ref{pos_ell}
and applying It{\^o}'s formula for variational solutions (see \cite[Theorem 4.2.5]{ConciseCourse}),
we obtain the following expression for the process $\zeta_{t}=\Vert X_{t}\Vert^{2}$:
\begin{multline*}
\zeta_{t}=\zeta_{0}+\int_{0}^{t}2\Vert X_{s}\Vert_{Q}dw_{s}+
\int_{0}^{t}(2\langle X_{s},F(X_{s})\rangle +
\mbox{tr\,} Q+2\langle AX_{t},X_{t}\rangle )ds\\
\leq\zeta_{0}+\int_{0}^{t}2\Vert X_{s}\Vert_{Q}dw_{s}+
\int_{0}^{t}(2\langle X_{s},F(X_{s})\rangle +\mbox{tr\,} Q)ds.
\end{multline*}
where we used the estimate $\langle Ax,x\rangle \leq0$. Similarly
to the derivation of the bound \eqref{eq:lower_bound}, we obtain 
$$
\int_{0}^{t}e^{-s}v(\omega,s)dw_{s}\geq-C,\qquad v(\omega,s):=2\Vert X_{s}\Vert_{Q}.
$$
 Step 1 is completed in exactly the same way as in proof of Theorem
\ref{pos_ell}. Next, we observe that the structure of the drift term
in the proof of Theorem \ref{pos_ell} has only been used in Lemma
\ref{nu->dalee}. Therefore, to complete the proof of Theorem \ref{lin_pos}
it suffices to prove an analogue of Lemma \ref{nu->dalee} in the
case $A\neq0$. Fix $X_{0}\in\mathcal{V}_{\infty}\bigl(H\bigr)$ and $K_{R}(a)$. Let
$\nu=\mbox{Law}(X_{0})$. Set
$$
\tau(R):=R\cdot(6\cdot(1+\sup_{x\in  H}\Vert F(x)\Vert))^{-1}.
$$

\begin{lemma}\label{nu->dalee-semigroup} Suppose that $X_{0}$ is independent
of $\bigl(W_{t},\, t\in\mathbb{R}_{+}\bigr)$ and $\mbox{supp}\,\nu\subset K_{R/2}(a)$. Then
\begin{equation}
\mathbb{P}\Bigl(X_{t}\in K_{R}(a)\Bigr)>0\quad\mbox{for all }t\in(0,\tau(R)].
\label{eq:taur-1}
\end{equation}
\end{lemma}
{\bf Proof of Lemma \ref{nu->dalee-semigroup}.} Note that 
the  
variational solution $X_{t}$ is also a mild solution to \eqref{eq:sde} 
(see \cite[F.0.5, F.0.6]{ConciseCourse}),
i.e. $$
X_{t}=S_{t}X_{0}+\int_{0}^{t}S_{t-s}I\, dW_{s}+\int_{0}^{t}S_{t-s}F(X_{s})ds.
$$
Clearly, it suffices to prove that for
all $t\in(0,\tau(R)]$
$$
\mathbb{P}(\Vert X_{t}-X_{0}\Vert_{Q}>R/2)<1,\,\,\mbox{if}\quad\mbox{Law}(X_{0})=\nu.
$$
 We have
 \begin{multline}
\mathbb{P}\Bigl(\Vert X_{t}-X_{0}\Vert_{Q}>\frac{R}{2}\Bigr)\leq\mathbb{P}\Bigl(\Bigl\Vert (S_{t}-I)X_{0}+
\int_{0}^{t}S_{t-s}I\, dW_{s}\Bigr\Vert_{Q}>\frac{R}{4}\Bigr)\\
+\mathbb{P}\Bigl(\Bigl\Vert \int_{0}^{t}S_{t-s}F(X_{s})ds\Bigr\Vert_{Q}>\frac{R}{4}\Bigr).
\label{eq:2ver-1}\end{multline}
 Since the semigroup $S_{t}$ is contracting,
$$
\Bigl\Vert \int_{0}^{t}S_{t-s}F(X_{s})ds\Bigr\Vert_{Q}\leq
\Bigl\Vert\int_{0}^{t}S_{t-s}F(X_{s})ds\Bigr\Vert\leq
\int_{0}^{t}\Vert S_{t-s}F(X_{s})\Vert ds
\leq\tau(R)\cdot\sup_{x\in  H}\Vert F(x)\Vert<\frac{R}{4},
$$
i.e. the second probability on the right-hand side of \eqref{eq:2ver-1}
is zero. Thus,
\begin{multline}
\mathbb{P}\Bigl(\Vert X_{t}-X_{0}\Vert_{Q}>\frac{R}{2}\Bigr)\leq
\mathbb{P}\Bigl(\Bigl\Vert (S_{t}-I)X_{0}+
\int_{0}^{t}S_{t-s}I\, dW_{s}\Bigr\Vert_{Q}>\frac{R}{4}\Bigr)\\
\leq\mathbb{P}\Bigl(\Vert(S_{t}-I)X_{0}+
\int_{0}^{t}S_{t-s}I\, dW_{s}\Vert>\frac{R}{4}\Bigr).
\label{eq:est-pr}\end{multline}
 The process $W_{A}=(W_{A}(t),\, t\geq0)$ given
by $W_{A}(t):=\int_{0}^{t}S_{t-s}I\, dW_{s}$ is called a stochastic
convolution. Since
$$
\int_{0}^{T}\mbox{tr}\, S(r)QS^{*}(r)dr=
\mbox{tr}\int_{0}^{T}\Vert S(r)\Vert_{Q}^{2}dr<\infty,
$$
 $W_{A}$ is an $\mathcal{F}_{t}$-adapted Gaussian random variable,
continuous in mean square, with the non-degenerate covariance operator
${\displaystyle \int_{0}^{t}\Vert S(r)\Vert_{Q}^{2}dr}$ (see \cite[Theorem 5.2]{Zabczyk}).
It can be easily seen that $(S_{t}-I)X_{0}$ and $W_{A}(t)$
are independent random variables. By the convolution formula
\begin{equation}
\mathbb{P}(\Vert(S_{t}-I)X_{0}+
\int_{0}^{t}S_{t-s}I\, dW_{s}\Vert\leq\frac{R}{4})=
\int_{H}\rho_{t}(U_{R/4}(0)-w)\sigma_{t}(dw),
\label{eq:w}
\end{equation}
where $\sigma_{t}=\mbox{Law}(S_{t}-I)X_{0}$ and $\rho_{t}=\mbox{Law}(W_{A}(t))$.
But the integrand is strictly positive by the properties
of the Gaussian random variable $\rho_{t}$ (see \cite[Theorem 3.5.1]{gaussian_measures}),
and $\sigma_{t}$ is a probability measure, hence \eqref{eq:w} is
a strictly positive quantity. Therefore, the right-hand side of \eqref{eq:est-pr}
is strictly less than $1$. This completes the proof of Lemma \ref{nu->dalee-semigroup}
and Theorem \ref{lin_pos}. \hfill{}$\square$
\end{proof}

\section*{Acknowledgements}

The author is grateful to V.I. Bogachev,  A.A. Novikov,
 A.Yu. Veretennikov and  A.D. Manita for fruitful discussions. 
 The work was partially supported by RFBR grants 14-01-00237 and 14-01-00319.

\end{document}